\documentclass[a4paper]{article}
    \usepackage{empheq} 
    \usepackage{bm} 
    \usepackage{mathtools} 
    \mathtoolsset{showonlyrefs=true} 
    \usepackage{graphicx} 
    \usepackage[hang,small,bf]{caption} 
    \usepackage[subrefformat=parens]{subcaption}
    \usepackage{pgfplots} 
    \usepackage{booktabs} 
    \usepackage{pgfplotstable}
    \pgfplotsset{filter discard warning=false}
    \usetikzlibrary{patterns}
    \usepackage{hyperref} 
    \usepackage{listings} 
    \usepackage{verbatim} 
    \usepackage{algorithm}
    \usepackage[noend]{algpseudocode} 

    \newcommand{\stt}{\State}
    \newcommand{\In}{\Require}
    \newcommand{\Out}{\Ensure}
    \usepackage{physics}
    \usepackage[hmargin=2cm, vmargin=2.25cm]{geometry}
    \usepackage{setspace} 
    \usepackage[T1]{fontenc}
    \usepackage[bitstream-charter]{mathdesign}
    \let\mathcal\mathscr
    \let\mathbb\undefined
    \usepackage[bb=boondox, scr=pxtx]{mathalfa}

    \usepackage{amsthm}
    \theoremstyle{definition}
    \newtheorem{Def}{Definition}[section]
    \newtheorem{Prop}{Proposition}[section]

    \definecolor{gray1}{HTML}{555555}
    \definecolor{white1}{HTML}{000000}
    \definecolor{black1}{HTML}{000000}

    \pgfplotscreateplotcyclelist{srkr}{
    {gray1},
    {black1},
    {white1}
    }

        \newcommand{\bbN}{\in \mathbb{N}}

        \newcommand{\bbC}{\in \mathbb{C}}
        
        \newcommand{\bbCnn}{\in \mathbb{C}^{n \times n}}

        \newcommand{\bCnn}{\mathbb{C}^{n\times n}}

        \newcommand{\proot}[2]{#1^{1/ #2}}

        \newcommand{\floor}[1]{\lfloor #1 \rfloor}
        
        \newcommand{\zerom}{O}
        
        \newcommand{\Order}{\mathcal{O}}

    \newcommand{\citerin}{\cite[Definition 4.17]{higham2008functions}}
    \newcommand{\citerinn}{\cite[Theorem 7.2]{higham2008functions}}
    \usepackage{authblk}
    \title{A cost-efficient variant of the incremental Newton iteration for the matrix $p$th root}
    \author{Fuminori Tatsuoka\thanks{f-tatsuoka@na.nuap.nagoya-u.ac.jp},\ \ Tomohiro Sogabe,\ Yuto Miyatake,\ Shao-Liang Zhang}
    \affil{\small \textit{Department of Computational Science and Engineering, Graduate School of Engineering, Nagoya University, Japan}}
    \providecommand{\keywords}[1]{\textbf{Keywords}  #1}
    \providecommand{\classificasion}[1]{\textbf{MR(2010) Subject Classification}  #1}
    \date{}

\begin{document}
    \maketitle
\vspace{-2.2em}
    \begin{abstract}
        Incremental Newton (IN) iteration, proposed by Iannazzo, is stable for computing the matrix $p$th root, and its computational cost is $\Order (n^3p)$ flops per iteration.
        In this paper, a cost-efficient variant of IN iteration is presented.
        The computational cost of the variant well agrees with $\Order (n^3 \log p)$ flops per iteration, if $p$ is up to at least 100.
    \end{abstract}
    \\
    \classificasion{65F30; 65F60; 65H04}
    \section{Introduction}

        A matrix $p$th root ($p\bbN$) of $A\bbCnn$ is defined as a solution of the following matrix equation:
        $$X^p=A.$$
        While this matrix equation might have infinitely many solutions, the target of this paper is a solution whose eigenvalues lie in the set $\qty{z\bbC\setminus\qty{0}:-\pi/p<\arg z<\pi/p}$.
        If $A$ has no nonpositive real eigenvalues, the target solution is unique \citerinn\ and is referred to as the principal matrix $p$th root of $A$, denoted by the symbol $A^{1/p}$.
        Throughout this paper, $A$ is assumed to have no nonpositive real eigenvalues.
        The principal matrix $p$th root arises
        in lattice quantum chromodynamics (QCD) calculations \cite{clark2007accelerating}
        and in the computation of the matrix logarithm \cite{higham2008functions}
        that corresponds to the inverse function of the matrix exponential.
        Therefore, numerical algorithms for computing the principal matrix $p$th root have been developed during the past decade.

        Numerical algorithms for the principal matrix $p$th root can be classified roughly into direct methods and iterative methods.
        Direct methods include, for example, the Schur method \cite{smith2003schur},
        the matrix sign method \cite{bini2005algorithms},
        and a method based on repeated eigenvalues of $A$  \cite{sadeghi2011computing}.
        The Schur method can be performed in $\Order (n^3p)$ flops,
        the matrix sign method can be performed in at least $\Order (n^3p\log p)$ flops,
        and the computational cost of the method based on repeated eigenvalues is not explicitly stated in \cite{sadeghi2011computing}.
        Therefore, in terms of computational cost, the Schur method is likely the method of choice for large-scale problems.
        Iterative methods include Newton's method and Halley's method for $A^{1/p}$,
        proposed by Iannazzo \cite{iannazzo2006newton,iannazzo2008family},
        and Newton's method for $A^{-1/p}$, proposed by Guo \cite{guo2006schur}.
        In this paper, we consider Newton's method for $A^{1/p}$, since that method is the most fundamental iterative method.
        In addition, it has been reported that Newton's method for $A^{1/p}$ gives a more accurate solution than the Schur method for some ill-conditioned matrices \cite{iannazzo2006newton}.

        Now, let us recall several results for Newton's method by Iannazzo \cite{iannazzo2006newton}.
        It is known that Newton's method for a matrix $p$th root can be written as
        \begin{align}\label{eq:ItrSN}
           X_{k+1} = \cfrac{(p-1)X_k+AX_k^{1-p}}{p}, \quad k=0,1,2,\dots,
        \end{align}
        with an initial guess $X_0$ satisfying $AX_0=X_0A$.
        However, it is not always guaranteed that this method converges to the principal $p$th root.
        Iannazzo showed that if both of the following conditions,
        \begin{empheq}[left=\empheqbiglbrace]{align}
            & \mathrm{all\ eigenvalues\ of\ }A
              \mathrm{\ lie\ in\ the\ set\ }\qty{z\bbC:\Re z >0,|z|\le 1}, \label{eq:condition1}\\
            & X_0=I, \label{eq:condition2}
        \end{empheq}
        are satisfied, then Newton's method \eqref{eq:ItrSN} converges to $A^{1/p}$.
        Next, Iannazzo proposed a preconditioning step,
        computing $\tilde{A}=A^{1/2}/\norm{A^{1/2}}$ with a consistent norm (say, $p$-norm, Frobenius norm),
        because then $\tilde{A}$ satisfies the condition \eqref{eq:condition1} for any $A$.
        Even if the matrix $A$ is preconditioned, Newton's iteration \eqref{eq:ItrSN} could be unstable in the neighborhood of $A^{1/p}$ \cite{smith2003schur}.
        Then, Iannazzo proposed three stable iterations:
        \begin{align} \label{eq:ItrIN}
           \begin{cases}
               X_{k+1} = X_k + H_k,\ F_k = X_kX_{k+1}^{-1},\\
               H_{k+1} = -\frac{1}{p}H_k \qty(
               \sum\limits_{i=0}^{p-2}(i+1)X_{k+1}^{-1}F_k^i
               )H_k,
           \end{cases}
           \quad\qty(X_0 = I,\ H_0 = \frac{A-I}{p})
        \end{align}
        \begin{align}\label{eq:Itr3.9}
           \begin{cases}
               X_{k+1} = X_k + H_k,\ F_k = X_kX_{k+1}^{-1},\\
               H_{k+1} = -X_k \qty(
               \frac{I-F_k^p}{p}+F_k^{p-1}(F_k-I)
               ),
           \end{cases}
           \quad\qty(X_0 = I,\ H_0 = \frac{A-I}{p})
        \end{align}
        and
        \begin{align} \label{eq:ItrHWA}
           \begin{cases}
               X_{k+1} = X_k\qty(\cfrac{(p-1)I+N_k}{p}),\\[1.5em]
               N_{k+1} = \qty(\cfrac{(p-1)I+N_k}{p})^{-p}N_k.
           \end{cases}
           \quad\qty\Big(X_0 = I,\ N_0 = A)
        \end{align}
        In particular,
        iteration \eqref{eq:ItrIN} is called incremental Newton (IN) iteration,
        and iteration \eqref{eq:ItrHWA} is called coupled Newton iteration.

        It is known that Newton's method converges quadratically in a neighborhood of the solution,
        but global convergence of that method is not guaranteed.
        One way to globalize the convergence of Newton's method is by using damping.\footnote{
            A damped Newton iteration is represented as $X_{k+1}=X_k+\alpha_kH_k\ (\alpha_k \in (0,1])$, where $\alpha_k$ is a relaxation factor chosen to reduce residuals.
        }
        From this point of view,
        it might be possible to apply damping to IN iteration \eqref{eq:ItrIN} and iteration \eqref{eq:Itr3.9}.
        Comparing these two iterations,
        the cost of IN iteration \eqref{eq:ItrIN} is $\Order (n^3p)$ flops per iteration, higher than $\Order (n^3\log p)$ flops for iteration \eqref{eq:Itr3.9}.
        On the other hand, the incremental part of IN iteration \eqref{eq:ItrIN} is computed in the form of $H_{k+1} = f_k(H_k)$,
        in contrast to iteration \eqref{eq:Itr3.9}.
        This characteristic of IN iteration \eqref{eq:ItrIN} might provide a new viewpoint for convergence analysis to confirm that $H_k$ converges to $\zerom$.
        That is to say, if $H_{k+1}$ explicitly includes $H_k$,
        then $H_{k+1}$ is represented as
        $H_{k+1}=\qty(f_k\circ f_{k-1}\circ \dots \circ f_0)(H_0)$,
        and its convergence behavior might be analyzed using composite mapping $\qty(f_k\circ f_{k-1}\circ \dots \circ f_0)$ and initial matrix $H_0$.
        Thus, IN iteration \eqref{eq:ItrIN} is worth considering.

        The purpose of this paper is to provide a cost-efficient variant of IN iteration \eqref{eq:ItrIN} whose increment part is computed in the form $H_{k+1}=f_k(H_k)$.
        In this paper, we reduce the cost of IN iteration \eqref{eq:ItrIN}
        by finding a specific matrix polynomial in IN iteration \eqref{eq:ItrIN}
        and proposing a decomposition of the matrix polynomial.

        The remainder of this paper is organized as follows.
        In section 2, a variant of IN iteration is shown, and we numerically estimate its cost at $\Order (n^3\log p)$ flops per iteration.
        In section 3, we present the results of numerical experiments.
        We conclude in section 4.

    \section{Variant of IN iteration}
        The computational cost for computing the increment part
        \begin{align} \label{eq:HkIN}
            H_{k+1}=-\frac{1}{p}H_k\qty(\sum_{i=0}^{p-2}(i+1)X_{k+1}^{-1}F_k^i)H_k
        \end{align}
        is the highest in IN iteration \eqref{eq:ItrIN},
        because $(2p+2/3)n^3+\Order (n^2)$ flops are required for Eq. \eqref{eq:HkIN},
        and $(2p+10/3)n^3+\Order (n^2)$ flops for IN iteration \eqref{eq:ItrIN}.
        In this section, without losing the previous matrix $H_k$,
        Eq. \eqref{eq:HkIN} is rewritten to reduce the number of matrix multiplications
        whose computational costs are $\Order (n^3)$ flops.

    \subsection{Rewriting the increment}
        From the definition of IN iteration \eqref{eq:ItrIN},
        the increment $H_k$ is equivalent to $X_{k+1}-X_k$, and thus
        \begin{align}
                H_kX_{k+1}^{-1}= (X_{k+1}-X_k)X_{k+1}^{-1}= I-F_k.
        \end{align}
        Substituting this relation into Eq. \eqref{eq:HkIN} yields
        \begin{align}
            H_{k+1} & =-\frac{1}{p}H_kX_{k+1}^{-1}\qty(
                \sum_{i=0}^{p-2}(i+1)F_k^i
                )H_k\\
            & = -\frac{1}{p}\qty(I-F_k)\qty(
                \sum_{i=0}^{p-2}(i+1)F_k^i
                )H_k\\
            & = -\frac{1}{p} \qty[
                I + F_k + F_k^2 +
                \dots + F_k^{p-2} - (p-1)F_k^{p-1}
                ]H_k\\
            & = -\frac{1}{p} \qty\bigg{
                \qty\Big[-(p-1)F_k+pI]\qty\Big[
                    I+F_k + F_k^2 + \dots + F_k^{p-2}
                    ]-(p-1)I
                }H_k.
            \label{eq:CeinHk}
        \end{align}
        Introducing the matrix polynomial
        $$P_d(X):= I+X+X^2+\dots+X^d,$$
        enables Eq. \eqref{eq:HkIN} to be simplified further to
        \begin{align}\label{eq:ProHk}
            H_{k+1} = -\frac{1}{p} \qty\bigg{
                \qty\Big[-(p-1)F_k+pI]P_{p-2}(F_k)-(p-1)I
                }H_k.
        \end{align}
        The number of matrix multiplications for Eq. \eqref{eq:ProHk} is equal to
        the number of matrix multiplications for $P_{p-2}(F_k)$ plus two.
        We now define a variant of IN iteration as
        \begin{align} \label{eq:ItrProp}
            \begin{cases}
            X_{k+1} = X_k + H_k,\ F_k = X_kX_{k+1}^{-1},\\
            H_{k+1} = -\frac{1}{p} \qty\bigg{
            \qty\Big[-(p-1)F_k+pI]P_{p-2}(F_k)-(p-1)I
            }H_k.
            \end{cases}
        \end{align}
        This new expression motivates us to reduce the number of matrix multiplications for computing $P_{p-2}(F_k)$.

        Furthermore, this variant \eqref{eq:ItrProp} is as stable as original IN iteration \eqref{eq:ItrIN}.
        We use the following definition of stability to analyze the variant \eqref{eq:ItrProp}.

        \begin{Def}[\citerin]
            \label{def:Stability}
            Consider an iteration $X_{k+1}=g(X_k)$ with a fixed point $X$.
            Assume that $g$ is Fr\'{e}chet differentiable at $X$.
            The iteration is stable in a neighborhood of $X$ if the Fr\'{e}chet derivative
            $L_g(X)$ has bounded powers, that is, there exists a constant $c$ such that $\|L_g^i(X)\|\le c$ for all $i > 0$.
        \end{Def}

        In Definition \ref{def:Stability}, $L_g^i(X)$ is $i$th power of the Fr\'{e}chet derivative $L$ at $X$.
        For more details of definitions of $L_g^i(X)$, $\|L_g^i(X)\|$, and other notations used for stability analysis, see Appendix.
        Then, we show that the variant \eqref{eq:ItrProp} is stable.


    \begin{Prop}
        The variant \eqref{eq:ItrProp} is stable.
    \end{Prop}
    \begin{proof}
        The iteration function for the variant \eqref{eq:ItrProp} is
        \begin{align}\label{eq:ItrFun}
                G\qty(\mqty[X\\H]) & = \mqty[
                    X+H\\
                    -\frac{1}{p}\qty\bigg{
                        \qty\Big[-(p-1)F+pI]P_{p-2}(F)-(p-1)I
                        }H
                    ]
                \quad \qty(F = X(X+H)^{-1}),
        \end{align}
        and the fixed point is $\smqty[A^{1/p}\\ \zerom ]$.
        In order to calculate the Fr\'{e}chet derivative of $G$ at $\smqty[A^{1/p}\\ \zerom ]$, we calculate $G\qty(\smqty[A^{1/p}\\ \zerom ])$ and $G\qty(\smqty[A^{1/p}+E_X\\ \zerom +E_H])$,
        where $\norm{E_X}$ and $\norm{E_H}$ are sufficiently small.
        Substituting $X=A^{1/p}$ and $H=\zerom$ into Eq. \eqref{eq:ItrFun},
        \begin{align}
             G\qty(\mqty[A^{1/p}\\ \zerom ])
                = \mqty[
                    A^{1/p}\\
                    -\frac{1}{p}\qty{
                        \qty\Big[-(p-1)I+pI]\qty\Big[\sum_{n=0}^{p-2}I]-(p-1)I
                        }\zerom
                    ]
                = \mqty[A^{1/p}\\ \zerom ],\label{eq:pp03}
        \end{align}
        and substituting $X=A^{1/p}+E_X$ and $H= \zerom +E_H$ into Eq. \eqref{eq:ItrFun},
        \begin{align}
            G\qty(\mqty[A^{1/p}+E_X\\ \zerom +E_H])
            & = \mqty[
                A^{1/p}+E_X+E_H\\
                -\frac{1}{p}\qty{
                    \qty\Big[-(p-1)F_{\Delta}+pI]\qty\Big[\sum_{i=0}^{p-2}F_{\Delta}^i]
                    -(p-1)I
                    }E_H
                ]\\
            &\hspace{1em} \quad \qty(F_{\Delta}=(A^{1/p}+E_X)(A^{1/p}+E_X+E_H)^{-1})\\[0.8em]
            & = \mqty[
                A^{1/p}+E_X+E_H\\
                -\frac{1}{p}\qty[
                    I + F_{\Delta} + F_{\Delta}^2 +
                    \dots + F_{\Delta}^{p-2} - (p-1)F_{\Delta}^{p-1}
                    ]E_H
                ]. \label{eq:pp02}
        \end{align}
        Since $\norm{E_X}$ and $\norm{E_H}$ are sufficiently small, $F_{\Delta}$ becomes

        \begin{align}
            F_{\Delta} & = (A^{1/p}+E_X)(A^{1/p}+E_X+E_H)^{-1}\\
            & = \qty[A^{1/p}+E_X]\qty[A^{-1/p}-A^{-1/p}(E_X+E_H)A^{-1/p}+\Order (\norm{E_X+E_H}^2)]\\
            & = I - E_HA^{-1/p} + \Order (\norm{E_X}^2) + \Order (\norm{E_H}^2) + \Order (\norm{E_X}\norm{E_H}). \label{eq:pp01}
        \end{align}
        Using Eq. \eqref{eq:pp01}, $F_{\Delta}^i$ becomes
        \begin{align}
            F_{\Delta}^i & = \qty(I - E_HA^{-1/p} + \Order (\norm{E_X}^2) + \Order (\norm{E_H}^2) + \Order (\norm{E_X}\norm{E_H}))^i\\
            & = I - iE_HA^{-1/p} + \Order (\norm{E_X}^2) + \Order (\norm{E_H}^2) + \Order (\norm{E_X}\norm{E_H}).
        \end{align}
        Therefore, the lower part of \eqref{eq:pp02} can be rewritten as
        \begin{align}\textstyle
            & -\frac{1}{p}\qty[I + F_{\Delta} + F_{\Delta}^2 +
                \dots + F_{\Delta}^{p-2} - (p-1)F_{\Delta}^{p-1}]E_H\\
            &\qquad = -\frac{1}{p}\Bigl[I + (I - E_HA^{-1/p}) + (I-2E_HA^{-1/p})+\dots + (I-(p-2)E_HA^{-1/p})\\
            &\qquad \qquad \qquad - (p-1)(I-(p-1)E_HA^{-1/p})
                + \Order (\norm{E_X}^2) + \Order (\norm{E_H}^2) + \Order (\norm{E_X}\norm{E_H})\Bigr]E_H\\
            &\qquad =  -\frac{1}{p}\qty[\frac{p(p-1)}{2}E_HA^{-1/p} + \Order (\norm{E_X}^2) + \Order (\norm{E_H}^2) + \Order (\norm{E_X}\norm{E_H})]E_H\\
            &\qquad = \Order (\norm{E_X}^2)+\Order (\norm{E_H}^2),
        \end{align}
        and we have
        \begin{align}
            & G\qty(\mqty[A^{1/p}+E_X\\ \zerom +E_H])
            = \mqty[
                A^{1/p}+E_X+E_H\\
                \Order (\norm{E_X}^2)+\Order (\norm{E_H}^2)
                ].\label{eq:pp04}
        \end{align}
        From Eq. \eqref{eq:pp03} and Eq.\eqref{eq:pp04}, it holds that
        \begin{align}
            G\qty(\mqty[A^{1/p}+E_X\\ \zerom +E_H]) - G\qty(\mqty[A^{1/p}\\ \zerom ]) -
            \mqty[I&I\\ \zerom &\zerom ]\mqty[E_X\\E_H]
            =\mqty[\zerom \\ \Order (\norm{E_X}^2)+\Order (\norm{E_H}^2)]
            = o\qty(\norm{\mqty[E_X\\E_H]}),
        \end{align}
        and we obtain
        \begin{align}
            L_G\qty(\mqty[A^{1/p}\\ \zerom ],\mqty[E_X\\E_H])
            = \mqty[I&I\\ \zerom &\zerom ]\mqty[E_X\\E_H].
        \end{align}
        The matrix $\smqty[I&I\\\zerom&\zerom]$ is idempotent because
        $$\mqty[I&I\\\zerom&\zerom]^2 = \mqty[I&I\\\zerom&\zerom].$$
        Then, for all $i>0$, $\norm{L_G^i\qty(\smqty[A^{1/p}\\ \zerom ])}$ is bounded.
        From the above,
        the variant \eqref{eq:ItrProp} is stable.\footnote{
            The stability of IN iteration \eqref{eq:ItrIN} can be proved in a similar manner.
        }
\end{proof}

        \color{black}

        In the next subsection, we provide a means of reducing matrix multiplications of $P_{p-2}(F_k)$.

    \subsection{Decomposition of the polynomial.}\label{sub:3.2}
        If $d\ge 3$, the matrix polynomial $P_d(X)$ can be rewritten in a more efficient form:
        \begin{align}\label{def:MapCein}
            & P_d(X)= \begin{cases}
            P_{\frac{d-1}{2}}(X^2) \cdot (X+I) & (d\mathrm{\ is\ odd}) \\
            P_{\frac{d-2}{2}}(X^2) \cdot (X^2 + X) + I & (d\mathrm{\ is\ even}).\\
            \end{cases}
        \end{align}

        On the right-hand side of Eq. \eqref{def:MapCein}, there is a new matrix polynomial
        whose variable is $X^2$ and degree is approximately half of $d$.
        This decomposition reduces the number of matrix multiplications by almost a factor of two.
        Thus, the number of matrix multiplications of $P_d(X)$ is reduced by applying the decomposition \eqref{def:MapCein} to $P_d(X)$ repeatedly.

        Let us show the example of $d=57$.\footnote{
            The polynomial $P_{57}(F_k)$ appears when calculating the matrix $59$th root.
            }
        \begin{align}
            P_{57}(X) & =I+X+X^2+\dots+X^{57}\label{eq:ceinExampleStt}\\
            & = \qty\Big{P_{28}(X^2)}\qty\Big{X+I}\\
            & = \qty\Big{
                P_{13}(X^4)(X^4+X^2)+I
                }\qty\Big{X+I}\\
            & = \qty\Big{
                P_{6}(X^8)(X^4+I)(X^4+X^2)+I
                }\qty\Big{X+I}\\
            & \qquad \vdots \qquad\\
            & = \qty\Big{
                \qty[
                    (X^{32}+X^{16}+I)(X^{16}+X^8)+I
                    ]\qty[X^4+I]\qty[X^4+X^2]+I
                }\qty\Big{X+I}.\label{eq:ceinExampleFin}
        \end{align}
        In this example, $P_{57}(X)$ of Eq. \eqref{eq:ceinExampleStt} is computed using 56 matrix multiplications by naive implementation.
        On the other hand, after applying the decomposition \eqref{def:MapCein} to Eq. \eqref{eq:ceinExampleStt} four times, Eq. \eqref{eq:ceinExampleFin} can be computed with nine matrix multiplications.
        In detail, five matrix multiplications are required for
        constructing five intermediate matrices, $X^2,X^4,X^8,X^{16}$, and $X^{32}$, and another four matrix multiplications are required for multiplication of the subpolynomials.

        Finally, we combine variant \eqref{eq:ItrProp} with decomposition \eqref{def:MapCein} into Algorithm \ref{alg:Algorithm1} for practice.
        \begin{algorithm}[htpb]
            \caption{
                Newton's method with the variant of IN iteration
                }\label{alg:Algorithm1}
            \begin{algorithmic}[1]
                \In $A\bbCnn$ (Satisfying condition \eqref{eq:condition1} in section 1), $p\bbN$
                \Out $X \approx \proot{A}{p}$
                \stt Decompose $P_{p-2}$ by applying the decomposition \eqref{def:MapCein} repeatedly.
                \stt $X_0 \leftarrow I$
                    ($\because$ Condition \eqref{eq:condition2})$,\ H_0 \leftarrow \frac{A-I}{p}$
                \For{$k = 0,1,2\dots$ until convergence}
                    \stt $X_{k+1} = X_k + H_k$
                    \stt $F_k = X_kX_{k+1}^{-1}$
                    \stt Compute $P_{p-2}(F_k)$
                    \stt $
                        H_{k+1} = -\frac{1}{p} \qty{
                            \qty[-(p-1)F_k+pI]P_{p-2}(F_k)-(p-1)I
                            }H_k$
                \EndFor
                \stt $X \leftarrow X_k$
            \end{algorithmic}
        \end{algorithm}

    \subsection{Estimation of the computational cost of the variant}
        We calculated the computational cost of the variant \eqref{eq:ItrProp} for $p\in[5,100]$ numerically and found that cost to be consistent with $(2\floor{2\log_2(p-1)}+8/3)n^3$.
        Here, the computational cost $8/3 n^3$ results from the computation of $F_k(=X_kX_{k+1}^{-1})$ by using the LU decomposition of $X_{k+1}$.
        While a proof that the cost of variant \eqref{eq:ItrProp} is $\Order (n^3 \log p)$ flops per iteration is left for future work,
        this numerical result agrees with that expectation.
        In addition, we calculated the costs of IN iteration \eqref{eq:ItrIN} and the iteration \eqref{eq:Itr3.9} for $p\in[5,100]$ to compare them with that of variant \eqref{eq:ItrProp}. The result is shown in Fig. \ref{fig:CalcCost}.
        \begin{figure}[htbp]\centering
            \includegraphics{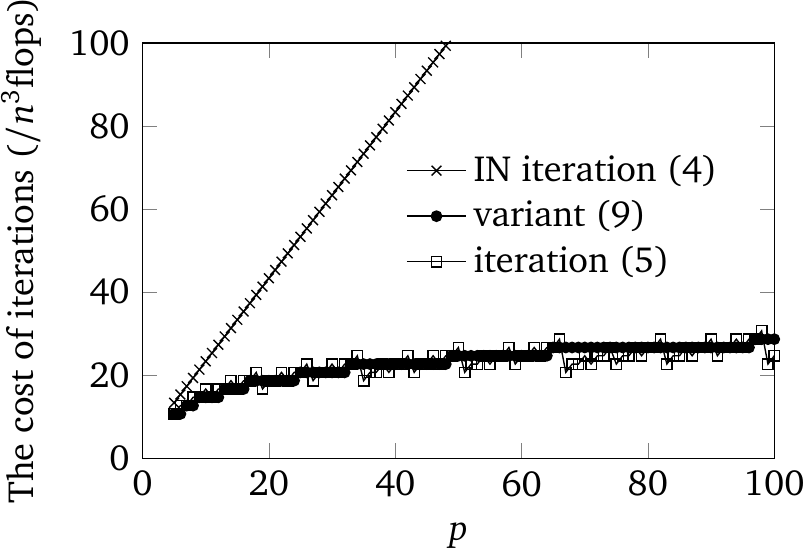}
            \caption{
                The computational costs per iteration for the three iterations
                }\label{fig:CalcCost}
        \end{figure}
        It is clear from the figure that the computational cost of the variant \eqref{eq:ItrProp} is lower than that of IN iteration \eqref{eq:ItrIN} and competitive with that of iteration \eqref{eq:Itr3.9}.
        For example, when $d=59$, the computational cost of variant \eqref{eq:ItrProp} is
        approximately a quarter of that of IN iteration \eqref{eq:ItrIN} and slightly higher than that of iteration \eqref{eq:Itr3.9}.

    \section{Numerical experiment}
        This section describes a numerical experiment in which
        the principal $59$th roots of test matrices are calculated.
        The test matrices are described in Table \ref{tab:numExpTestMat}.
        \begin{table}[htbp] \centering \caption{Test matrices.}
            \begin{tabular}{lrrrl}
                \toprule
                Test matrix $A$ (Matrix ID) & Size  & Non-zero elements & cond($A$)   & Symmetry Property \\
                \midrule
                msc01440\cite{davis2011university} (1) & 1440  & 44998 & $3.3 \times 10^{6\ \ }$ & Symmetric positive define \\
                Random matrix (2) & 1500  & 2250000 & $3.8 \times 10^{2\ ~}$ & Symmetric positive define \\
                NNC1374\cite{boisvert1997matrix} (3) & 1374  & 8606  & $3.7 \times 10^{14}$ & Unsymmetric \\
                \bottomrule
            \end{tabular} \label{tab:numExpTestMat}
        \end{table}
        First, we preconditioned the test matrices
        to satisfy the sufficient condition \eqref{eq:condition1} of global convergence in section 1:
        all eigenvalues of $A$ lie in the set $\qty{z\bbC:\Re z >0,|z|\le 1}$.
        Thus, we computed $\tilde{A} = A^{1/2}/\|A^{1/2}\|_F$.
        Then, we computed $\tilde{A}^{1/59}$ by IN iteration \eqref{eq:ItrIN}, variant \eqref{eq:ItrProp} of Algorithm \ref{alg:Algorithm1}, and iteration \eqref{eq:Itr3.9}.
        The computational costs of these three iterations are shown in Table \ref{tab:numExpCos}.
        \begin{table}[htbp] \centering \caption{Computational costs for computing the principal 59th root.}
            \begin{tabular}{lr}
                \toprule
                Iteration & Computational costs per iteration(flops) \\
                \midrule
                IN iteration \eqref{eq:ItrIN} & $(118+10/3)n^3+\Order (n^2)$\\
                variant \eqref{eq:ItrProp} &  $(22+8/3)n^3+\Order (n^2)$\\
                iteration \eqref{eq:Itr3.9} & $(20+8/3)n^3+\Order (n^2)$ \\
                \bottomrule
            \end{tabular} \label{tab:numExpCos}
        \end{table}
        For this experiment, Python 3.5 was used for programming,
        and Intel$^{\mathrm{(R)}}$ Core$^{\mathrm{TM}}$ i7 2.8GHz CPU and 8GB RAM were used for computation.

        First, Figure \ref{fig:numexp1} shows the ratios of computation time of these three iterations. From Fig. \ref{fig:numexp1}, the computation time of variant \eqref{eq:ItrProp} is approximately one fourth of that of IN iteration \eqref{eq:ItrIN} and slightly longer than that of \eqref{eq:Itr3.9} in all cases.
        Here it can be seen that both the computation time and the computational cost decreased.
        \pgfplotsset{width=7cm,height=5.8cm}
        \begin{figure}[htbp] \centering
            \includegraphics{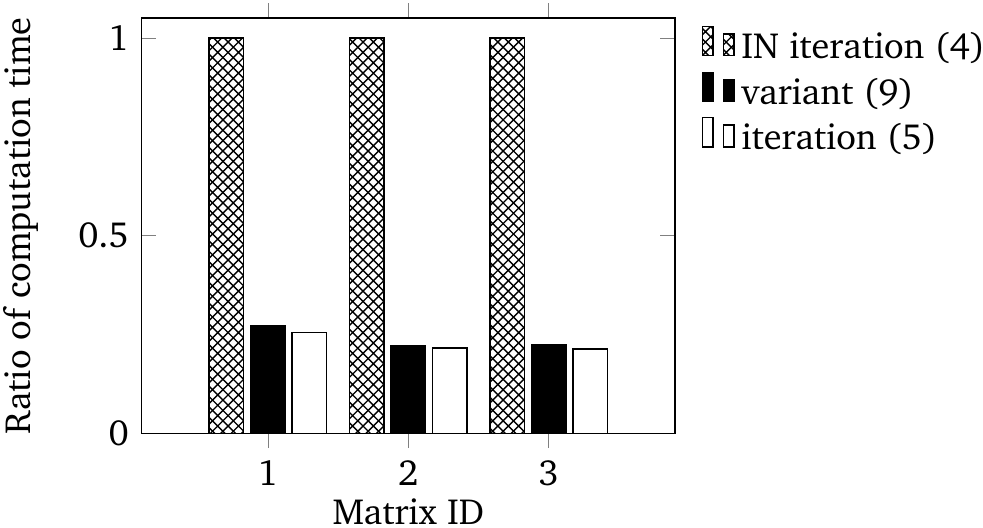}
            \caption{
                Time comparison of the three iterations
                }\label{fig:numexp1}
        \end{figure}
        Next, Figure \ref{fig:numexp2} shows the relative residual defined as $R(X)=\norm{X^p-A}_F/\norm{A}_F$ for these three iterations.
        The figure shows that the convergence behavior of variant \eqref{eq:ItrProp} differs little from that of IN iteration \eqref{eq:ItrIN} and iteration \eqref{eq:Itr3.9}.
        Since there is some possibility of numerical cancellation of variant \eqref{eq:ItrProp}, IN iteration \eqref{eq:ItrIN} is slightly better than variant \eqref{eq:ItrProp} in terms of accuracy.
        \begin{figure}[htbp]\centering
            \includegraphics{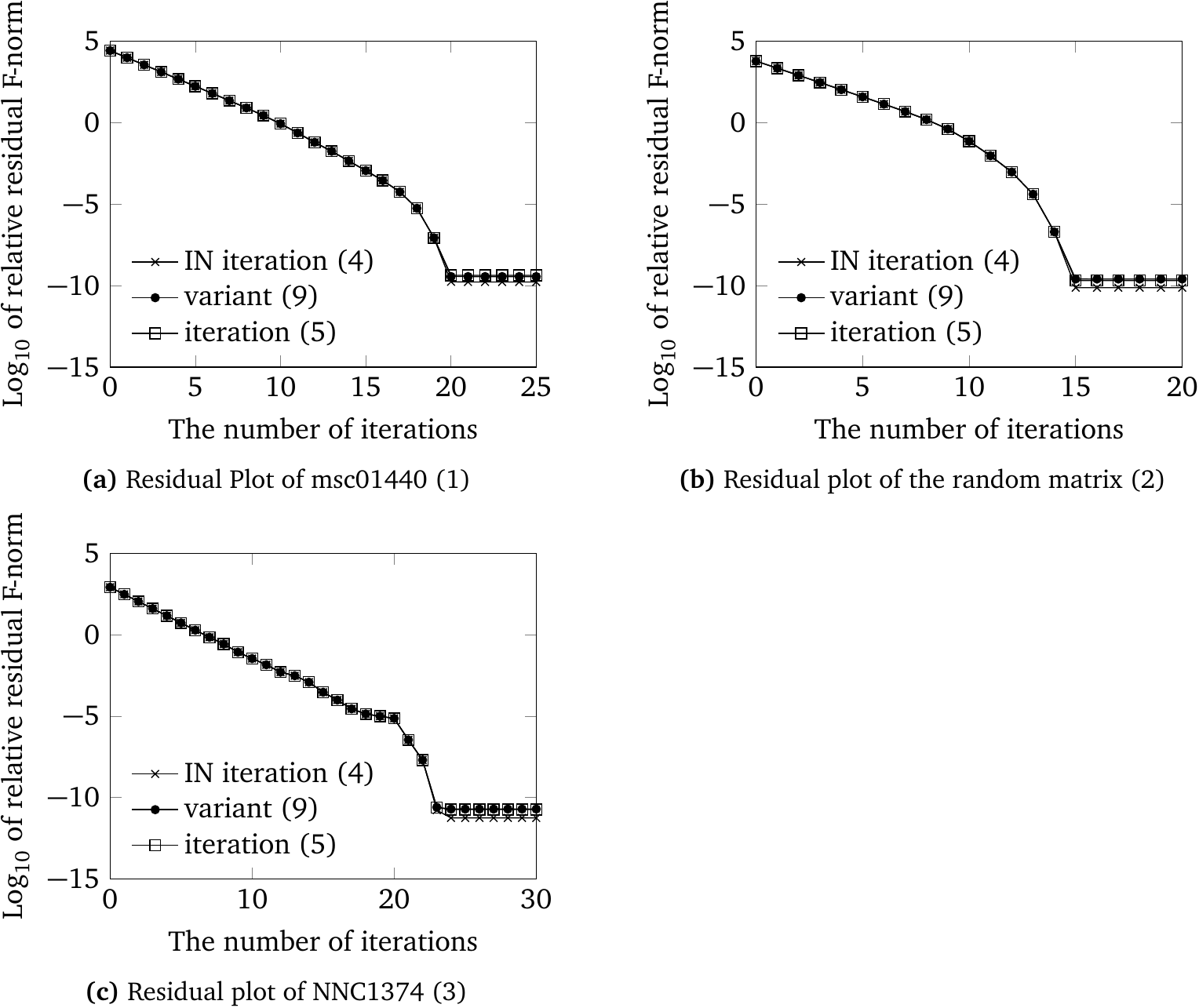}
            \caption{
                Residual comparison of the three iterations
                }\label{fig:numexp2}
        \end{figure}

    \section{Conclusion and future work}
        In this paper, a variant of IN iteration is proposed
        whose computational cost well agreed with $\Order (n^3\log p)$ flops per iteration if $p$ is up to at least 100,
        and whose increment part still has the form $H_{k+1}=f_k(H_k)$.
        We have learned from the results of the numerical experiment that
        the variant is competitive with iteration \eqref{eq:Itr3.9} in terms of accuracy and computation time.
        The proposed variant therefore becomes a choice for practical application.

        The most important future work is to prove that the computational cost of the variant is $\Order (n^3\log p)$. Other future work includes reducing the computation time of Newton's method for the principal matrix $p$th root by reducing the number of iterations.
        However, it is not clear how to choose a better initial guess than the conventional initial guess $I$ (the identity matrix).
        It might be easier to find a good initial guess, when considering the damped Newton method.

    \section*{Acknowledgment}
        The authors are grateful to the reviewer for the careful reading and the comments that substantially enhanced the quality of the manuscript.
        This work has been supported in part by JSPS KAKENHI (Grant No. 26286088).\\
        \\

    \color{black}
    \section*{Appendix}

        In this section,
        we recall some definitions and notations which were given in \cite{higham2008functions},
        where we consider the matrix norm is consistent.
\begin{enumerate}

    \item The notation $X=\Order (\norm{E})$ denotes that $\norm{X}\le c\norm{E}$ for some constant $c$ for all sufficiently small $\norm{E}$, while $X=o(\norm{E})$ means that $\norm{X}/\norm{E}\to 0$ as $E\to \zerom$ \cite[p. 321]{higham2008functions}.

    \item The Fr\'{e}chet derivative of a matrix function $f:\bCnn\to\bCnn$ at a point $X\bbCnn$ is a linear mapping
    \begin{align}
        \bCnn &\overset{L}{\to} \bCnn\\
        E &\mapsto L(X,E)
    \end{align}
    such that for all $E\bbCnn$
    \begin{align}
        f(X+E)-f(X)-L(X,E)= o(\norm{E}).
    \end{align}
    If we need to show the dependence on $f$ we will write $L_f(X,E)$.
    When we want to refer to the mapping at $X$ and not its value in a particular direction we will write $L(X)$ \cite[p. 56]{higham2008functions}.

    \item The norm of $L(X)$ is defined by$ \|L(X)\| := \underset{Z\ne \zerom}{\max}\frac{\|L(X,Z)\|}{\|Z\|}$ \cite[p. 56]{higham2008functions}.

    \item We write $L^i(X)$ to denote the $i$th power of the Fr\'{e}chet derivative $L$ at $X$,
    defined as $i$-fold composition; thus $L^3(X,E)\equiv L\bigl(X,L(X,L(X,E))\bigr)$ \cite[p. 97]{higham2008functions}.
\end{enumerate}

\end{document}